\documentclass[11pt,reqno]{amsart}

\usepackage[left=3.5cm,right=3.5cm,top=3.5cm,bottom=3.5cm]{geometry}
\usepackage{tipa}
\usepackage{amssymb}
\usepackage{graphicx}
\usepackage[hidelinks]{hyperref}
\usepackage{enumerate}
\usepackage{xcolor}
\usepackage{amsmath}
\usepackage{tikz}
\usepackage{cancel}
\allowdisplaybreaks[4]

\numberwithin{equation}{section}
\newtheorem{theorem}{Theorem}[section]

\newtheorem{corollary}[theorem]{Corollary}
\newtheorem{lemma}[theorem]{Lemma}

\theoremstyle{definition}

\newtheorem{remark}[theorem]{Remark}

\numberwithin{equation}{section}

\title{Bijections between $t$-Core Partitions and $t$-Tuples}
\date{March 17, 2019}
\author[H. Zhong]{Hao Zhong}
\address{(H. Zhong) The School of Information Science and Technology, Chengdu University
of Technology, Chengdu, 610059, China}
\email{11435011@zju.edu.cn}

\keywords{core partitions, simultaneous cores, cores with distinct parts}
\subjclass[2010]{05A17, 11P81}

\begin{document}

\maketitle

\thispagestyle{empty}

\begin{abstract}
This note introduces some bijections relating core partitions and tuples of integers. We apply these bijections to count the number of cores with various types of restriction, including fixed number of parts, limited size of parts, parts divisible by some integer, and distinct parts. For example, we prove that the number of $2t$-core partitions into $l$ even parts equals the number of $t$-core partitions into $l$ parts. We also generalize one expression for simultaneous cores, which was given by Baek, Nam and Yu, recently. Subsequently, we use this expression to obtain recurrence satisfied by numbers of $(s,s+1,\cdots,s+r)-$core partitions for $s\geq1$.
\end{abstract}

\section{Introduction}

A partition of a positive integer $n$ is a nonincreasing sequence of positive integers $\lambda_1$, $\lambda_2$, $\cdots$, $\lambda_l$ with sum $n$. We denote $(\lambda_1$, $\lambda_2$, $\cdots$, $\lambda_l)$ by $\lambda$, and we also write $\lambda=(1^{\#_1}2^{\#_2}\cdots)$ where $\#_i=\#_i(\lambda)$ denotes the number of $\lambda_j$ that equals $i$. Each $\lambda$ is associated to a Ferrers diagram $[\lambda]$, which is a pattern of dots, with the $j$-th row having $\lambda_j$ dots. For each node, its hook is defined as the set consisting of the node itself, and all the nodes to the right of it in the same row together with the nodes below it in the same column. By counting the number of nodes in the hook, we define the hook length of a node. A hook of length $t$ is called a $t$-hook and it is obvious that there exist $t$ different types of $t$-hook. The diagram containing no $t$-hook is called a $t$-core, and the corresponding partition is called $t$-core partition.

Gordon and Kerber \cite{James1981} showed that by removing all the rim $t$-hooks from $[\lambda]$, we always finish up with the same $t$-core, i.e., each diagram has a uniquely determined $t$-core. This explains why such a diagram is called a $t$-core. They also proved that $[\lambda]$ is a $t$-core if and only if there is no dot whose hook length is divisible by $t$.

Core partitions of numerous types of additional restrictions have long been studied, since they are closely related to the representation of symmetric group \cite{James1981}, the theory of cranks \cite{Garvan1990}, Dyck-paths \cite{Amdeberhan2015, Anderson2002, Yang2015}, and Euler's theorem \cite{Straub2016}. To solve core problems, mathematicians provide many different tools, including $t$-abacus \cite{Anderson2002, James1981}, Hasse diagram \cite{Yan2016, Yang2015} and even ideas from quantum mechanics \cite{Johnson2015}.

This paper includes new proofs of some known theorems, which lead to several new results on core partitions. By giving different expressions for partitions, we aim to evaluate the cardinality of various types of core partitions.

Section 2 reviews three bijections relating $t$-cores and $t$-tuples, and also includes their applications. Furthermore, we count the number of core partitions in which the number of parts is limited, as well as the number of core partitions in which all parts are divisible by a fixed integer. As a consequence, we surprisingly find that the number of $2t$-core partitions into $l$ even parts equals the number of $t$-core partitions into $l$ parts. In Section 3, inspired by a recent expression for simultaneous core partitions (\cite{Baek2019}), we give a new representation for simultaneous cores, which is useful in studying the topics on $(s,s+1,\cdots,s+r)$-cores. In the last section, we conclude this note with some formulas for the core partitions with distinct parts.

\section{Bijections}

First, we set the notation. In this article, $r$, $s$ and $t$ represent positive integers. For a partition $\lambda=(\lambda_1,\lambda_2,\cdots,\lambda_l)$, let $\#(\lambda)=l$ and $\sigma(\lambda)=\sum_{i=1}^{l}\lambda_i$. Let $\mathcal{P}$ be the set of all partitions and $\mathcal{P}_t$ the set of all $t$-core partitions. We also introduce some notations from $q$-series, $$(a;q)_{k}:=\prod_{n=0}^{k-1}(1-aq^n)$$ and $$(a_1, a_2, \ldots, a_n,;q)_{k}:=\prod_{j=1}^{n}(a_j,q)_{k}.$$ Hence, $$\sum_{\lambda\in\mathcal{P}}q^{\sigma(\lambda)}=\frac{1}{(q;q)_{\infty}}.$$

The first bijection we have to mention was given by James and Kerber \cite{James1981}, answering how a core partition is derived from a regular partition.

\begin{theorem}[\cite{James1981}]\label{th:bij1}
There is a bijection $\Phi$: $\mathcal{P} \to \mathcal{P}^t \times \mathcal{P}_t$, $$\Phi(\lambda)=(\lambda^{(0)}, \lambda^{(1)}, \cdots, \lambda^{(t-1)}, \lambda^{(t)})$$ such that $\sigma(\lambda)=t\sum_{j=0}^{t-1}\sigma(\lambda^{(j)})+\sigma(\lambda^{(t)})$. Hence,
\begin{equation}\label{eq:gen1}
\sum_{\lambda\in\mathcal{P}_t}q^{\sigma(\lambda)}=\frac{(q^t;q^t)^{t}_{\infty}}{(q;q)_{\infty}}.
\end{equation}
\end{theorem}

(\ref{eq:gen1}) makes it possible to study the properties of core partitions by analytic methods. With the help of Ramanujan's theta function and modular forms, authors obtained many remarkable results (see \cite{Chern2016, Hirschhorn2009, Nath2013, Ono1995}).

The second bijection is derived from the first one by utilizing $t$-abacus.
\begin{theorem}[\cite{Anderson2002}, \cite{Olsson2007}, \cite{Xiong2015}, \cite{Xiong2016}]\label{th:bij2}

Let $\beta(\lambda)$ denote the set of hook lengths in the first column of $[\lambda]$, i.e., $\beta(\lambda)=\{\lambda_{j}+\#(\lambda)-j,1\leq j \leq \#(\lambda)\}$. Then there is a bijection $\Phi_1$: $\mathcal{P}_{t}$ $\to$ $\mathbb{Z}_{\geq0}^{t-1}$, $$\Phi_1(\lambda)=(n_0=0,n_1,\cdots,n_{t-1})$$ where $n_{j}=|\{a\in\beta(\lambda):a\equiv j\mod t\}|$ for $0\leq j\leq t-1$. Moreover, $\#(\lambda)=\sum_{j=0}^{t-1}n_{j}$ and $\sigma(\lambda)=\sum_{j=0}^{t-1}\Big(jn_{j}+t\binom{n_{j}}{2}\Big)-\binom{\#(\lambda)}{2}$. Hence,
\begin{equation}\label{eq:gen2}
\sum_{\lambda\in\mathcal{P}_t}q^{\sigma(\lambda)}=\sum_{n_{1},\ldots,n_{t-1}\geq0}q^{\sum_{j=1}^{t-1}(jn_{j}+t\binom{n_{j}}{2})-\binom{\sum_{j=1}^{t-1} n_{j}}{2}}.
\end{equation}
\end{theorem}

We say $\mu$ is generated by $\lambda$, if $\#_{j+1}(\mu)=\#_{j}(\lambda)$ for any $j\geq1$, namely, $\lambda_{i}=\mu_{i}-1$ if $\mu_{i}>1$, and we write $\lambda=gen(\mu)$. Let $\phi_{j}$ be a map: $\mathcal{P}$ $\to$ $\mathcal{P}$ such that $\phi_{j}(\lambda)=(1^{j},2^{\#_{1}(\lambda)},3^{\#_{2}(\lambda)},\cdots)$. Then $\lambda=gen(\mu)$ if and only if $\mu=\phi_{j}(\lambda)$ for some nonnegative integer $j$. Moreover, ($\phi_{j}(\lambda))_{1}=\lambda_{1}+1$, $\#(\phi_{j}(\lambda))=\#(\lambda)+j$, and $\sigma(\phi_{j}(\lambda))=\sigma(\lambda)+\#(\lambda)+j$. Notice that $\lambda$ is a $t$-core partition only if $gen(\lambda)$ is a $t$-core partition. We can give a proof of Theorem \ref{th:bij2} without $t$-abacus.

\begin{proof}[Proof of Theorem \ref{th:bij2}]
For a fixed $\lambda\in\mathcal{P}$, the first column of $[\lambda]$ is determined, so $\Phi_{1}$ is well-defined. $[\lambda]$ is a $t$-core where there is no hook length divisible by $t$, so $n_0=0$. As for the number of parts, $\#(\lambda)=|\beta(\lambda)|=\sum_{j=0}^{t-1}n_{j}$.

For two $t$-core partitions with the same image in $\mathbb{Z}_{\geq0}^{t-1}$, they share the same number of parts. Thus, if we prove that the restriction of $$\Phi_1\text{: }\{t\text{-core partitions into }l\text{ parts}\}\text{ }\to\text{ }\{(n_1,n_2,\cdots,n_{t-1})\in\mathbb{Z}_{\geq0}^{t-1}:\sum_{j=1}^{t-1}n_{j}=l\}$$ is bijective, then the bijection statement of the theorem is a straightforward corollary.

Let $\Phi_1|_{l}$ denote the restricted map. We prove the claim above by induction on $l$.

If $l=1$, then $\{t$-core partitions with one part$\}=\{(j),1\leq j\leq t-1\}$ and its image is $\{\Phi_1((j)),1\leq j\leq t-1\}=\{(0,\cdots,0,n_{j}=1,0,\cdots,0),1\leq j\leq t-1\}$. Therefore, $\Phi_1|_{1}$ is a bijection.

Suppose $\Phi_1|_{s}$ is bijective for $s<l$.

We first prove $\Phi_1|_{l}$ is injective. Let $\lambda$ and $\mu$ be two different $t$-core partitions into $l$ parts such that $\Phi_{1}(\lambda)=\Phi_{1}(\mu)$, $\lambda=\phi_{j}(gen(\lambda))$ and $\mu=\phi_{k}(gen(\mu))$ for some $0\leq j$, $k\leq t-1$.

Let $\mathcal{N}_{t}=\{(n_0,n_1,\cdots,n_{t-1})\in\mathbb{Z}_{\geq0}^{t}:n_0=0\}.$ Let $\varphi_{0}$ and $\varphi_{1}$ be two operators on $\mathbb{Z}_{\geq0}^{t}$ such that $$\varphi_{0}(n_{0},n_{1},\cdots,n_{t-1})=(n_{t-1},n_{0},n_{1},\cdots,n_{t-2}),$$
$$\varphi_{1}(n_{0},n_{1},\cdots,n_{t-1})=(n_{0},n_{1}+1,n_{2},\cdots,n_{t-1}).$$ Let $\psi_{j}=(\varphi_{1}\circ\varphi_{0})^{j}\circ\varphi_{0}$ for $j\geq0$, namely, $$\psi_{j}(n_0,\cdots,n_{t-1})=(n_{t-j-1},n_{t-j}+1,\cdots,n_{t-1}+1,n_{0},n_1,\cdots,n_{t-j-2}).$$ Then $\psi_{j}$ is a bijection from $\mathcal{N}_{t}\bigcap\{n_{t-1-j}=0\}$ to $\mathcal{N}_{t}\bigcap\{n_{j+1}=0,n_{i}>0$ for $0<i<j+1\}$ for $0\leq j\leq t-2$, and $\psi_{t-1}$ is a bijection from $\mathcal{N}_{t}$ to $\mathcal{N}_{t}\bigcap\{n_{i}>0$ for $i>0\}$. Thus
$\Phi_{1}(\phi_{j}(\lambda))=\psi_{j}(\Phi_{1}(\lambda))$ and $\phi_{j}(\mathcal{N}_{t})\bigcap\phi_{k}(\mathcal{N}_{t})=\varnothing$ if $j\neq k$.

Hence $\psi_{k}(\Phi_{1}(gen(\mu)))=\Phi_{1}(\mu)=\Phi_{1}(\lambda)=\psi_{j}(\Phi_{1}(gen(\lambda))).$ Moreover, $j=k$ and $\Phi_{1}(gen(\mu))=\Phi_{1}(gen(\lambda))$. If $\#(gen(\lambda))<l$, then $gen(\lambda)= gen(\mu)$ by assumption. Thus, $\lambda=\mu=\phi_{j}(gen(\lambda))$. This is a contradiction. If $\#(gen(\lambda))=l$, then $j=0$, so we can consider $gen(\cdots gen(gen(\lambda))\cdots)$ such that its number of parts is less than $l$. With a similar discussion, one can obtain $\lambda=\mu$. This is a contradiction.

Next, we prove that $\Phi_1|_{l}$ is surjective. For any $\textbf{n}=(n_0,n_1,\cdots,n_{t-1})\in\mathcal{N}_{t}$ with sum $l$, if $n_{j+1}=0$ and $n_{1}$, $n_{2}$, $\ldots$, $n_{j}>0$ for some $0\leq j\leq t-2$, then $\textbf{n}^{(j)}:=(n_{j+1},n_{j+2},\cdots,n_{t-1},n_{0},n_{1}-1,n_{2}-1,\cdots,n_{j}-1)$ is contained in $\mathcal{N}_{t}$, $|\textbf{n}^{(j)}|=l-j$ and $\psi_{j}(\textbf{n}^{(j)})=\textbf{n}.$ Otherwise, $n_{j}>0$ for $1\leq j\leq t-1$, so $\textbf{n}^{(t-1)}=(n_{0},n_{1}-1,n_{2}-1\cdots,n_{t-1}-1)\in\mathcal{N}_{t}$, $|\textbf{n}^{(t-1)}|=l-t+1$ and $\psi_{t-1}(\textbf{n}^{(t-1)})=\textbf{n}$. Hence for $1\leq j\leq t-1$, $\Phi_1(\phi_{j}(\Phi^{-1}_{1}(\textbf{n}^{(j)})))=\textbf{n}$. In case of $j=0$, we consider $\textbf{n}':=(n_{0},n_{1}-n_{m},n_{2}-n_{m},\cdots,n_{t-1}-n_{m})$ where $n_{m}$ is the minimum among $n_{1},\ldots,n_{t-1}$. Similarly, we have there exists $\lambda\in\mathcal{N}_{t}$ such that $\Phi_{1}(\lambda)=\textbf{n}'$, so $\Phi_{1}(\phi_{t-1}^{n_{m}}(\lambda))=\textbf{n}.$

Finally, we prove that $\sigma(\lambda)=\sum_{i=0}^{t-1}\Big(in_{i}+t\binom{n_{i}}{2}\Big)-\binom{\#(\lambda)}{2}$. It is obvious when $\#(\lambda)=1$. Assume that it holds for $\#(\lambda)<l$. Recall the proof of $\Phi_{1}$ is surjective. If $\Phi_{1}(gen(\lambda))=\textbf{n}^{(j)}$ for some $0\leq j\leq t-2$, then
\begin{eqnarray*}
  \sigma(\lambda) &=& \sum_{i=0}^{t-2-j}\Big(in_{i+j+1}+t\binom{n_{i+j+1}}{2}\Big)+\sum_{i=1}^{j}\Big((t-j+i-1)(n_{i}-1)+t\binom{n_{i}-1}{2}\Big)\\
  &&-\binom{\#(\lambda)-j}{2} \\
   &=& \sum_{i=0}^{t-1}\Big(in_{i}+t\binom{n_{i}}{2}\Big)-\binom{\#(\lambda)}{2}.
\end{eqnarray*}
When $\Phi_{1}(gen(\lambda))=\textbf{n}^{(t-1)}$, we can also obtain this equation by considering $\textbf{n}'$.
\end{proof}

\begin{remark}
(1) If we regard the empty partition $\emptyset$ as a $t$-core partition for any positive integer $t$, then we obtain an ordered tree of $t$-cores with root $\emptyset$. For example, when $t=3$, we have

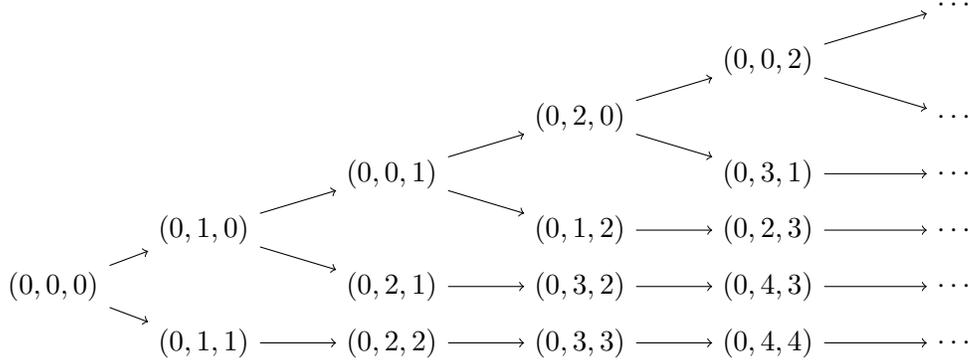
\usetikzlibrary{trees,arrows}
\tikzstyle{level 1}=[level distance=20mm]
\tikzstyle{level 2}=[level distance=25mm]
\tikzstyle{level 3}=[level distance=25mm]
\tikzstyle{level 4}=[level distance=25mm]
\tikzstyle{level 5}=[level distance=25mm]
\begin{tikzpicture}[grow=right,->]
  \node {$(0,0,0)$}
    child {node {$(0,1,1)$}
      child {node {$(0,2,2)$}
       child {node {$(0,3,3)$}
        child {node {$(0,4,4)$}
         child {node {$\cdots$}
         }
        }
       }
      }
    }
    child {node {$(0,1,0)$}
      child {node{$(0,2,1)$}
       child {node {$(0,3,2)$}
        child {node {$(0,4,3)$}
         child {node {$\cdots$}
         }
        }
       }
      }
      child {node{$(0,0,1)$}
       child {node {$(0,1,2)$}
        child {node {$(0,2,3)$}
         child {node {$\cdots$}
         }
        }
       }
       child {node {$(0,2,0)$}
        child {node {$(0,3,1)$}
         child {node {$\cdots$}
         }
        }
        child {node {$(0,0,2)$}
         child {node {$\cdots$}
         }
         child {node {$\cdots$}
         }
        }
       }
      }
    };
\end{tikzpicture}

 \noindent where each vertex is the image of one $t$-core in $\mathcal{N}_{t}$ and the children of the vertex represent the cores generated by that vertex.

(2) $\phi_{t-1}^{j}(\emptyset)$ is a $t$-core partition for $j\geq0$, so there are infinitely many $t$-cores for $t\geq2$.

(3) Due to the fact $\sigma(\lambda)=\sum_{i=0}^{t-1}\Big(in_{i}+t\binom{n_{i}}{2}\Big)-\binom{\sum n_{i}}{2}$, we have $\beta(\lambda)=\{(i-1)t+j, 1\leq i\leq n_{j}, 1\leq j\leq t-1\}$. In particular, $a\in\beta(\lambda)$ and $a>t$ only if $a-t\in\beta(\lambda)$.
\end{remark}

Recall the proof of Theorem \ref{eq:gen2}. It helps us to obtain some properties of core partitions with restriction on the number of parts.

\begin{theorem}\label{th:numberofparts}
Let $l$ be a positive integer. Then we have the following.

(1) The number of $t$-core partitions into $l$ parts is $\binom{l+t-2}{l}.$

(2) The number of $t$-core partitions into at most $l$ parts is $\binom{l+t-1}{l}.$

(3) The number of $t$-core partitions with $l$ parts greater than $1$ and $j$ parts equal to $1$ is
\begin{equation*}
\begin{cases}
\binom{l+t-3}{l}&\text{ if }0\leq j\leq t-2,\\
\binom{l+t-2}{l}&\text{ if }j=t-1,\\
0&\text{ if }j>t-1.
\end{cases}
\end{equation*}

(4) The number of $t$-core partitions into $l$ parts such that the smallest part equals $i$ is $\binom{l+t-i-2}{l-1}$.
\end{theorem}

\begin{proof}
Recall that the number of ways to put $n$ indistinguishable balls into $k$ distinguishable boxes is $\binom{n+k-1}{n}$, namely, $|\{(a_1,a_2,\cdots,a_{k})\in\mathbb{Z}_{\geq 0}:\sum_{i=1}^{k}a_{i}=n\}|=\binom{n+k-1}{n}$.

(1) Since $\Phi_{1}|_{l}$ is a bijection, $|\mathcal{P}_t\bigcap\{\#(\lambda)=l\}|=|\{(n_0,n_1,\cdots,n_{t-1})\in\mathbb{Z}^{t}_{\geq 0}:n_0=0,\sum_{i=1}^{t-1}n_{i}=l\}|=\binom{l+t-2}{l}$.

(2) By (1), we have the number of $t$-core partitions into at most $l$ parts equals $\sum_{j=0}^{l}\binom{j+t-2}{j}=\binom{l+t-1}{l}.$

(3) Let $\lambda$ be a $t$-core partition with $l$ parts greater than $1$ and $j$ parts equal to $1$. Then $gen(\lambda)$ is a $t$-core partition into $l$ parts such that $\Phi_{1}(gen(\lambda))$ satisfies $n_{t-1-j}=0$.

If $0\leq j\leq t-2$, then $|\{t$-core partition with $l$ parts greater than $1$ and $j$ parts equal to $1$ $\}|$ $=$ $|\{(n_0,n_1,\cdots,n_{t-1})\in\mathbb{Z}_{\geq 0}:n_0=n_{t-1-j}=0,\sum_{i=1}^{t-1}n_{i}=l\}|=\binom{l+t-3}{l}.$

If $j=t-1$, then $|\{t$-core partition with $l$ parts greater than $1$ and $t-1$ parts equal to $1$ $\}|$ $=$ $|\mathcal{P}_t\bigcap\{\#(\lambda)=l\}|=\binom{l+t-2}{l}.$

Obviously, $\#_{j}(\lambda)<t$ for any $t$-core partition $\lambda$, so the number is zero if $j>t-1$.

(4) Recall the previous remark: $a\in\beta(\lambda)$ and $a>t$ only if $a-t\in\beta(\lambda)$. Thus, we have $|\mathcal{P}_{t}\bigcap\{\#(\lambda)=l$, $\lambda_l=i\}|=|\mathcal{N}_{t}\bigcap\{\sum_{j=0}^{t-1}n_j=l,$ $n_{i}\geq 1,$ $n_{j}=0$ for $j<i\}|=\binom{l+t-i-2}{l-1}.$
\end{proof}

We say a partition is the conjugate of $\lambda$, denoted by $\lambda'$, if $\#(\lambda')=\lambda_1$ and $\lambda'_j=\sum_{i=j}^{\lambda_1}\#_{i}(\lambda)$ for $1\leq j \leq \#(\lambda')$. Hence, $[\lambda]$ and $[\lambda']$ transform into each other when reflected along the main diagonal. Moreover, $[\lambda]$ is a $t$-core if and only if $[\lambda']$ is a $t$-core. Hence we obtain the following corollary which was proved by Baek, Nam, and Yu \cite{Baek2019}.

\begin{corollary}[\cite{Baek2019}]
(1) The number of $t$-core partitions with largest part $x$ is $\binom{x+t-2}{x}.$

(2) The number of $t$-core partitions with largest part $x$ and second largest part $y$ is
\begin{equation*}
\begin{cases}
\binom{y+t-3}{y}&\text{ if }x-y<t-1,\\
\binom{y+t-2}{y}&\text{ if }x-y=t-1.
\end{cases}
\end{equation*}

(3) The number of $t$-core partitions such that the largest part equals $x$ and there are exactly $i$ parts equal to $x$ is $\binom{x+t-2-i}{t-1-i}$.
\end{corollary}

We have studied the core partitions in which either the number of parts or the largest part is limited. We now turn to the core partitions satisfying them both. Let $\mathcal{P}_{t;l;x}$ be the set of $t$-core partitions into $l$ parts, none exceeding $x$. Then we have a explicit but messy formula for $|\mathcal{P}_{t;l;x}|$.

\begin{theorem}\label{th:both}
$$|\mathcal{P}_{t;l;x}|=\binom{l+t-2}{l}+\sum_{r=1}^{t-1}(-1)^{r}\sum_{1\leq i_{1}<\cdots<i_{r}\leq t-1}\binom{l+t-2-\sum_{j=1}^{r}m_{i_j}}{t-2}$$
where
\begin{equation*}
m_{j}=\begin{cases}
2+\lfloor\frac{x}{t}\rfloor&\text{ if }1\leq j\leq r_{t}(x),\\
1+\lfloor\frac{x}{t}\rfloor&\text{ if }r_{t}(x)< j<t
\end{cases}
\end{equation*}
and $r_{t}(x)$ is the remainder of $x$ divided by $t$.
\end{theorem}

\begin{proof}
Let $A=\mathcal{N}_{t}\bigcap\{\sum_{j=0}^{t-1}n_{j}=l\}$, $A_{k}=A\bigcap\{n_{k}\geq m_{k}\}$ for $1\leq k\leq t-1$. Then $|\mathcal{P}_{t;l;x}|=|A-\bigcup_{k=1}^{t-1}A_{k}|$. By the principle of inclusion-exclusion, we obtain the result.
\end{proof}

There are not too many papers concerning core partitions with parts divisible by a fixed positive integer $d$. In the following theorem, we claim some interesting facts about this kind of partitions.

\begin{theorem}\label{th:mod}
Let $_{d}\mathcal{P}_t$ be the set of $t$-core partitions with each part a multiple of $d$. Then for $0\leq k \leq d-1$, there is a bijection $\Psi:$ $\mathcal{P}_t\to$ $_{d}\mathcal{P}_{dt-k}$, $\lambda\mapsto\Psi(\lambda)$ such that $\#(\lambda)=\#(\Psi(\lambda))$.

\end{theorem}

Before we prove this theorem, we give a useful lemma.

\begin{lemma}\label{newadded}
Let $P(\lambda)$ be the number of $t$-core partitions generated by $\lambda$. Then it equals the number of zero coordinates in $\textbf{n}=\Phi_{1}(\lambda)$. Furthermore, $\{P(\mu):$ $gen(\mu)=\lambda\}=\{1,2,\cdots,P(\lambda)\}$.
\end{lemma}

\begin{proof}
By the definition of $\phi_{j}$, $\phi_{j}(\lambda)$ is a $t$-core partition if and only if $n_{t-1-j}=0$. Thus, there exist exactly $P(\lambda)$ $t$-core partitions generated by $\lambda$. To find all these $t$-cores, we check the coordinates with index from $t-1$ to $0$. Once $n_{j}=0$, we obtain a $t$-core $\phi_{t-1-j}(\lambda)$. It is not hard to see that except for the first new finding that has the same $P$-value with $\lambda$, every new found $t$-core has a $P$-value $1$ less than the previous one. Therefore, $\{P(\mu):$ $gen(\mu)=\lambda\}=\{1,2,\cdots,P(\lambda)\}$.
\end{proof}

\begin{proof}[Proof of Theorem \ref{th:mod}]

Recall that $\lambda$ is a $t$-core partition if and only if its conjugate $\lambda'$ is a $t$-core partition. Let $E_{t,x}=\{\lambda\in\mathcal{P}_{t}:\lambda_1=x\}$ and $F_{t,x}=\{\lambda\in E_{t,x}:$ $d|\#_{j}(\lambda)$ for $j\geq0\}$. Then it is sufficient to prove that there is a bijection $\Psi':$ $E_{t,x}\to F_{dt-k,x}$, $\lambda\mapsto\Psi'(\lambda)$ such that $\lambda_1=\Psi(\lambda)_1$.

We begin with the case $x=0$. It is straightforward to define $\Psi'(\emptyset)=\emptyset$ where $\emptyset$ is the empty partition.

Notice that all core partitions can be generated by $\emptyset$ step by step. Let $\lambda$ be a $t$-core partition. When $\lambda\in$ $_{d}\mathcal{P}_t$, define $P_{d}(\lambda)=|\{\mu\in$ $_{d}\mathcal{P}_t:$ $gen(\mu)=\lambda\}|$. By Lemma \ref{newadded}, $P=P_{1}$. Since $\lambda\in F_{t,x}$ only if $\lambda=\phi_{jd}(gen(\lambda))$ for some $0\leq j\leq (t-1)/d$, we have $P_{d}(\lambda)=|\{0\leq j\leq (t-1)/d:$ $n_{t-1-jd}=0\}|$ and $\{P_{d}(\mu):$ $gen(\mu)=\lambda\}=\{1,2,\cdots,P_{d}(\lambda)\}$, so $P_{d}$ is distinct for each partition with the same generator. Taking $t=dt-k$ into the equation implies $P_{d}(\lambda)=|\{0\leq j\leq t-1:$ $n_{t-1-jd}=0\}|$. Hence, $\bigcup_{x\geq 0}E_{t,x}$ and $\bigcup_{x\geq 0}F_{dt-k,x}$ have the same tree structures.

When $x>0$, we define $\Psi'(\lambda)=\mu$ if $\Psi'(gen(\lambda))=gen(\mu)$ and $P(\lambda)=P_{d}(\mu)$. Then it is automatic that $\Psi'$ is a bijection.
\end{proof}

To see the same tree structure between the two types of partitions above, we set an example of $3$-core partitions and $5$-core partitions with even parts.

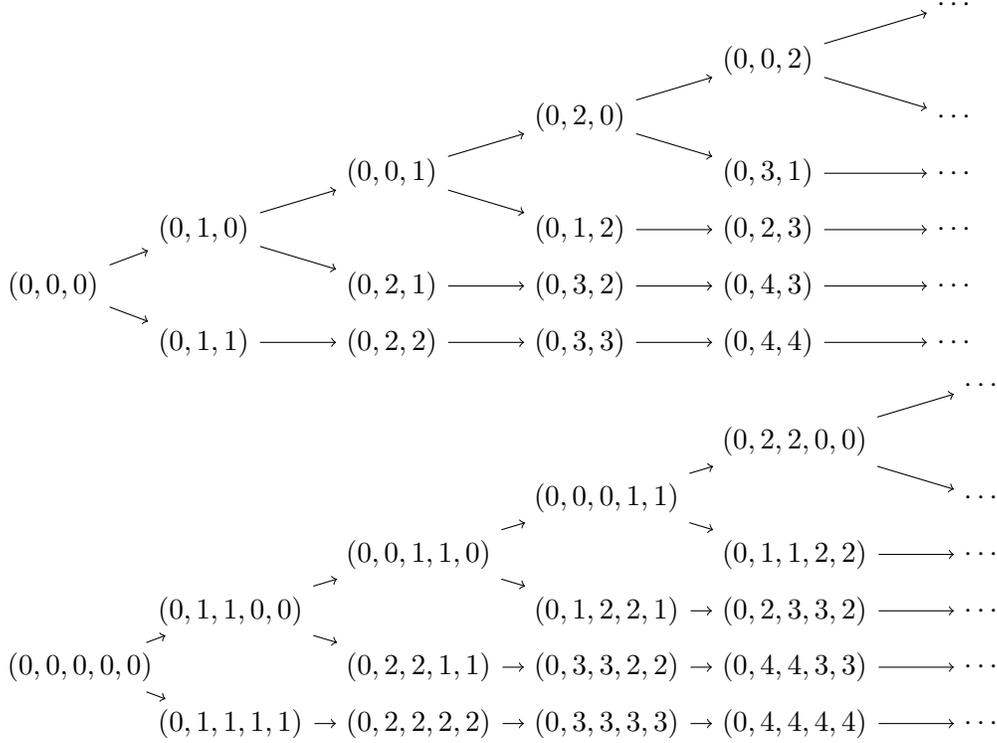
\usetikzlibrary{trees,arrows}
\tikzstyle{level 1}=[level distance=20mm]
\tikzstyle{level 2}=[level distance=25mm]
\tikzstyle{level 3}=[level distance=25mm]
\tikzstyle{level 4}=[level distance=25mm]
\tikzstyle{level 5}=[level distance=25mm]
\begin{tikzpicture}[grow=right,->]
  \node {$(0,0,0)$}
    child {node {$(0,1,1)$}
      child {node {$(0,2,2)$}
       child {node {$(0,3,3)$}
        child {node {$(0,4,4)$}
         child {node {$\cdots$}
         }
        }
       }
      }
    }
    child {node {$(0,1,0)$}
      child {node{$(0,2,1)$}
       child {node {$(0,3,2)$}
        child {node {$(0,4,3)$}
         child {node {$\cdots$}
         }
        }
       }
      }
      child {node{$(0,0,1)$}
       child {node {$(0,1,2)$}
        child {node {$(0,2,3)$}
         child {node {$\cdots$}
         }
        }
       }
       child {node {$(0,2,0)$}
        child {node {$(0,3,1)$}
         child {node {$\cdots$}
         }
        }
        child {node {$(0,0,2)$}
         child {node {$\cdots$}
         }
         child {node {$\cdots$}
         }
        }
       }
      }
    };
\end{tikzpicture}

\usetikzlibrary{trees,arrows}
\tikzstyle{level 1}=[level distance=20mm]
\tikzstyle{level 2}=[level distance=25mm]
\tikzstyle{level 3}=[level distance=25mm]
\tikzstyle{level 4}=[level distance=25mm]
\tikzstyle{level 5}=[level distance=25mm]
\begin{tikzpicture}[grow=right,->]
  \node {$(0,0,0,0,0)$}
    child {node {$(0,1,1,1,1)$}
      child {node {$(0,2,2,2,2)$}
       child {node {$(0,3,3,3,3)$}
        child {node {$(0,4,4,4,4)$}
         child {node {$\cdots$}
         }
        }
       }
      }
    }
    child {node {$(0,1,1,0,0)$}
      child {node{$(0,2,2,1,1)$}
       child {node {$(0,3,3,2,2)$}
        child {node {$(0,4,4,3,3)$}
         child {node {$\cdots$}
         }
        }
       }
      }
      child {node{$(0,0,1,1,0)$}
       child {node {$(0,1,2,2,1)$}
        child {node {$(0,2,3,3,2)$}
         child {node {$\cdots$}
         }
        }
       }
       child {node {$(0,0,0,1,1)$}
        child {node {$(0,1,1,2,2)$}
         child {node {$\cdots$}
         }
        }
        child {node {$(0,2,2,0,0)$}
         child {node {$\cdots$}
         }
         child {node {$\cdots$}
         }
        }
       }
      }
    };
\end{tikzpicture}

Combining Theorem \ref{th:mod} and Theorem \ref{th:numberofparts} leads to the following corollaries.

\begin{corollary}
The number of $2t$-core partitions into $l$ even parts equals the number of $t$-core partitions into $l$ parts.
\end{corollary}

\begin{corollary}
For $0\leq k\leq d-1$, the number of $(dt-k)$-core partitions into $l$ parts, each divisible by $d$ is $\binom{l+t-2}{l}$.
\end{corollary}

The third bijection was given by Garvan \cite{Garvan1990} and Johnson \cite{Johnson2015}, independently.

\begin{theorem}\label{th:bij3}
There is a bijection $\Phi_2$: $\mathcal{P}_t\to \{(c_{0},c_2,\ldots,c_{t-1})\in\mathbb{Z}^{t}:\sum_{j=0}^{t-1}c_{j}=0\}$, $\Phi_2(\lambda)=\textbf{c}$ such that
$$\sigma(\lambda)=\sum_{j=0}^{t-1}(jc_{j}+\frac{t}{2}c_{j}^{2}).$$ Hence,
\begin{equation}\label{eq:gen3}
\sum_{\lambda\in\mathcal{P}_t}q^{\sigma(\lambda)}=\sum_{\substack{c_{0},\ldots,c_{t-1}\in\mathbb{Z}\\c_{0}+\ldots+c_{t-1}=0}}q^{\sum_{j=0}^{t-1}(jc_{j}+\frac{t}{2}c_{j}^{2})}.
\end{equation}
\end{theorem}

We now set a transformation that maps $\mathcal{N}_{t}$ to $\{(c_{0},c_1,\ldots,c_{t-1})\in\mathbb{Z}^{t}:\sum_{j=0}^{t-1}c_{j}=0\}:$
\begin{equation*}
c_{j}=\begin{cases}c^*-n_{m-j} &\text{ if } 0\leq j\leq m,\\
c^*-1-n_{m-j+t} &\text{ if } m< j <t,
\end{cases}
\end{equation*}
where $c^*=1+\lfloor\frac{\sum_{i=0}^{t-1}n_{i}-1}{t}\rfloor$ and $m=r_{t}(\sum_{i=0}^{t-1}n_{i}-1)$, the remainder when $\sum_{i=0}^{t-1}{n_{i}}-1$ is divided by $t$. Then it is easy to verify the claim of Theorem \ref{th:bij3}. Besides, we have $c^*=max\{c_{j}\}$ and $m=max\{j:c_{j}=c^*\}$, so the corresponding partition $\lambda$ satisfies $\#(\lambda)=t(c^*-1)+m+1.$

The map $\Phi_{2}$ has been proved an effective tool to deal with the self-conjugate core partition problems \cite{Baek2019, Garvan1990, Johnson2015}. Below is one example.

\begin{lemma}\label{lem:1}(\cite{Baek2019})
Let $\textbf{c}=\Phi_{2}(\lambda)$. Then $\lambda$ is self-conjugate if and only if $c_{i}+c_{t-1-i}=0$ for each $0\leq i\leq t-1$.
\end{lemma}

Garvan, Kim, and Stanton gave the generating function for self-conjugate cores in \cite{Garvan1990} by utilizing the properties of $\Phi$ and $\Phi_{2}$. In the rest of this section, we rebuild the function by applying the famous Jacobi's triple product identity (one can find a quick proof in \cite{Andrews1998}), that is $$\sum_{n\in\mathbb{Z}}z^{n}q^{n^{2}}=(-zq;q^{2})_{\infty}(-z^{-1}q;q^{2})_{\infty}(q^2;q^2)_{\infty}$$ for $z\neq0$ and $|q|<1$. Let $\mathcal{S}_{t}$ be the set of self-conjugate $t$-core partitions. Then

\begin{theorem}(Garvan, Kim, and Stanton)
\begin{equation}\label{eq:gen-self}
\sum_{\lambda\in\mathcal{S}_{t}}q^{\sigma(\lambda)}=\begin{cases}
(q^{2t};q^{2t})_{\infty}^{t/2}(-q;q^{2})_{\infty}&\text{ if }t\text{ is even},\\
(q^{2t};q^{2t})_{\infty}^{(t-1)/2}(-q;q^{2})_{\infty}/(-q^t;q^{2t})_{\infty}&\text{ if }t\text{ is odd}.
\end{cases}
\end{equation}
\end{theorem}

\begin{proof}
According to (\ref{eq:gen3}) and Lemma \ref{lem:1}, we have
\begin{eqnarray*}
  \sum_{\lambda\in\mathcal{S}_{t}}q^{\sigma(\lambda)} &=& \sum_{\substack{c_{0},\ldots,c_{t-1}\in\mathbb{Z}\\c_{i}+c_{t-1-i}=0}}q^{\sum_{j=0}^{t-1}(jc_{j}+\frac{t}{2}c_{j}^{2})}\\
  &=& \prod_{j=0}^{\lfloor t/2\rfloor-1}\sum_{c_{j}+c_{t-1-j}=0}q^{jc_j+(t-1-j)c_{t-1-j}+\frac{t}{2}c_{j}^{2}+\frac{t}{2}c_{t-1-j}^{2}}\\
  &=& \prod_{j=0}^{\lfloor t/2\rfloor-1}\sum_{c_{j}\in\mathbb{Z}}q^{tc_{j}^{2}+(2j-t+1)c_{j}}\\
  &=& \prod_{j=0}^{\lfloor t/2\rfloor-1}(-q^{2j+1};q^{2t})_{\infty}(-q^{2t-2j-1};q^{2t})_{\infty}(q^{2t};q^{2t})_{\infty}.
\end{eqnarray*}

If $t$ is even, then
\begin{eqnarray*}
\sum_{\lambda\in\mathcal{S}_{t}}q^{\sigma(\lambda)}&=& (q^{2t};q^{2t})_{\infty}^{t/2}\prod_{j=1}^{t}(-q^{2j-1};q^{2t})_{\infty}\\
&=& (q^{2t};q^{2t})_{\infty}^{t/2}(-q;q^{2})_{\infty}.
\end{eqnarray*}

If $t$ is odd, then
\begin{eqnarray*}
\sum_{\lambda\in\mathcal{S}_{t}}q^{\sigma(\lambda)}&=& (q^{2t};q^{2t})_{\infty}^{(t-1)/2}\prod_{j=1}^{(t-3)/2}(-q^{2j-1};q^{2t})_{\infty}\prod_{j=(t+1)/2}^{t}(-q^{2j-1};q^{2t})_{\infty}\\
&=& (q^{2t};q^{2t})_{\infty}^{(t-1)/2}(-q;q^{2})_{\infty}/(-q^t;q^{2t})_{\infty}.
\end{eqnarray*}

\end{proof}

\section{simultaneous core partitions}

In this section, we fix $a$, $b$ two positive integers relatively prime, and let $\mathcal{P}_{t_1,t_2,\ldots,t_k}=\bigcap_{i=1}^{k}\mathcal{P}_{t_i}$. It is well known that $\mathcal{P}_{a,b}$ and $(a,b)$-Dyck path have the same cardinality (see \cite{Anderson2002}; for more generalized versions, see \cite{Amdeberhan2015, Yang2015}).

\begin{theorem}\label{th:two}(Anderson)
There is a one-to-one correspondence between partitions which are both $a$- and $b$- cores and shortest paths from $(0,b)$ to $(a,0)$ in $S:=\{(i,j):$ $0\leq i\leq a$ and $0\leq j\leq \lfloor(ab-ib)/a\rfloor\}$. Furthermore, $|\mathcal{P}_{a,b}|=Cat_{a,b}=\frac{1}{a+b}\binom{a+b}{a}$, the generalized Catalan number.
\end{theorem}

Inspired by the application of bijection $\Phi_2$ in \cite{Baek2019} and \cite{Johnson2015}, we now reprove the result by using $\Phi_1$. Before the proof, we claim two facts about $\mathcal{P}_{a,b}$.

\begin{lemma}\label{lem:3.1}
Rewrite $\mathcal{N}_t$ as $\{(n_{j})_{j\geq0}:$ $n_0=0$, $n_{j}\in\mathbb{Z}_{\geq0}$, $n_{j+t}=n_{j}$ for $j\geq0\}.$ For any $\lambda\in\mathcal{P}_{a}$ with image $\textbf{n}$ in $\mathcal{N}_a$, it is also a $b$-core partition if and only if $n_{j+b}-n_{j}\leq\lfloor\frac{b+r_{a}(j)}{a}\rfloor$ for $j\geq0$.
\end{lemma}

\begin{proof}
By the remark in Section 2, $\lambda$ is also contained in $\mathcal{P}_{b}$ if and only if for any $x\in\beta(\lambda)$, $x-b$ is contained in $\beta(\lambda)$. For any $j\geq0$,
$$(n_{j+b}-1)a+r_{a}(j+b)=max\{x\in\beta(\lambda): x\equiv j+b\pmod{a}\}.$$ Hence, $\lambda$ is also contained in $\mathcal{P}_{b}$ if and only if
\begin{eqnarray*}
  (n_{j+b}-1)a+r_{a}(j+b)&\leq& b+max\{x\in\beta(\lambda): x\equiv j\pmod{a}\},\\
  (n_{j+b}-1)a+r_{a}(j+b)&\leq&b+(n_{j}-1)a+r_{a}(j),\\
  a(n_{j+b}-n_{j})&\leq&b+r_{a}(j)-r_{a}(j+b),\\
  n_{j+b}-n_{j}&\leq&\lfloor\frac{j+b}{a}\rfloor-\lfloor\frac{j}{a}\rfloor,\\
  n_{j+b}-n_{r_{a}(j)}&\leq&\lfloor\frac{r_{a}(j)+b}{a}\rfloor,\\
  n_{j+b}-n_{j}&\leq&\lfloor\frac{r_{a}(j)+b}{a}\rfloor.
\end{eqnarray*}
\end{proof}

The lemma above is a twin of Lemma 3.1 in \cite{Johnson2015}, so one can also prove it by the bijection from $\mathcal{N}_{t}$ to $\{(c_{0},c_1,\ldots,c_t-1)\in\mathbb{Z}^{t}:\sum_{j=0}^{t-1}c_{j}=0\}$, mentioned in Section 2.

\begin{lemma}\label{lem:3.2}
The following sets are equivalent.\\
(1) $\mathcal{P}_{a,b}$,\\
(2) $\mathcal{N}_{a}\bigcap\{n_{j+b}-n_{j}\leq\lfloor\frac{r_{a}(j)+b}{a}\rfloor$ for $j\geq0\}$,\\
(3) $\mathcal{X}_{a,b}:=\{(x_{j})_{j\geq0}:$ $x_0=0$, $x_{j}\in\mathbb{Q}_{\geq0}$, $x_{j}\equiv j/a\pmod{1}$, $x_{j+a}=x_{j}$ and $x_{j+b}-x_{j}\leq b/a$ for $j\geq0\}$,\\
(4) $\mathcal{Z}_{a,b}:=\{(z_{j})_{j\geq0}:$ $z_{j}\in\mathbb{Z}_{\geq0}$, $\sum_{l=j}^{j+a-1}z_{l}=b$, and $\sum_{l=0}^{j-1}z_{l}\leq jb/a$ for $j\geq0\}$,\\
(5) $\mathcal{Y}_{a,b}:=\{(y_1,y_2,\cdots,y_{a-1})\in\mathbb{Z}^{a-1}:$ $y_1\leq y_2\leq\cdots\leq y_{a-1}$ and $0\leq y_{j}\leq jb/a$ for $1\leq j\leq a-1\}$.
\end{lemma}

\begin{proof}
Lemma \ref{lem:3.1} provides the bijection between $(1)$ and $(2)$.

To figure out the relation among the sets, we write down some maps:

$(2)\to(3):$ $x_{j}=n_{j}+r_{a}(j)/a$ for $j\geq 0$;

$(3)\to(4):$ $z_{j}=x_{jb}-x_{(j+1)b}+b/a$ for $j\geq 0$;

$(4)\to(5):$ $y_{j}=\sum_{l=0}^{j-1}z_{l}$ for $1\leq j\leq a-1$;

$(5)\to(1):$ $n_0=0$, $n_{jb}=\lfloor jb/a\rfloor-y_{j}$ for $1\leq j\leq a-1$ .

It is easy to check that all the maps are bijections, which completes the proof.
\end{proof}

Note that $\mathcal{Y}_{a,b}$ is an expression for $S$ such that $(i,j)=(i+1,y_{a-i})$ for $1\leq i\leq a-1$. Therefore, the first half of Theorem \ref{th:two} is a corollary to Lemma \ref{lem:3.2}. To count the number of elements of $S$, one may refer to \cite{Dvoretzky1947} or \cite{Bizley1954}.

Krewaras \cite{Kreweras1965} gave an order in partitions where zero parts are allowed. For two partitions allowing parts of size zero, $\lambda$ and $\mu$, denote $\mu\leq\lambda$ if $\mu_{j}\leq\lambda_{j}$ for any $j\geq1$. Due to the theorem in \cite{Kreweras1965}, the number of partitions not greater than $\lambda$ is $det(\binom{\lambda_{j}+1}{j-i+1})$. Taking $\lambda=(\lfloor (a-1)b/a\rfloor,\lfloor (a-2)b/a\rfloor,\cdots,\lfloor b/a\rfloor)$, the number is exactly the cardinality of $\mathcal{Y}_{a,b}$. Thus, we have $$det\Bigg(\binom{\lfloor jb/a\rfloor+1}{j-i+1}_{(a-1)\times(a-1)}\Bigg)=\frac{1}{a+b}\binom{a+b}{a}.$$ Amdeberhan \cite{Amdeberhan2015} used this formula to obtain some identities for Catalan numbers.

Besides, Lemma \ref{lem:3.2} leads to many interesting corollaries.

\begin{corollary}\label{cor}
Let $\varphi$ be the operator on $\{(z_j)_{j\geq0}:z_{j}=z_{r_{a}(j)}\}$, induced from $\varphi_{0}$ such that the $j$th coordinate of $\varphi(\textbf{z})$ is $z_{r_{a}(j-1)}$. Let $\textbf{z}\in\mathcal{Z'}_{a,b}:=\{(z_j)_{j\geq0}:$ $\sum_{l=0}^{a-1}z_{l}=b$, $z_j\in\mathbb{Z}_{\geq0}$, and $z_{j}=z_{r_{a}(j)}$ for $j\geq0\}$. Then there exists a unique element $\tau$ in the group $\{\varphi^j,0\leq j\leq a-1\}$ such that $\tau(\textbf{z})\in\mathcal{Z}_{a,b}$.
\end{corollary}

\begin{proof}
If $\textbf{z}$ happens to be contained in $\mathcal{Z}_{a,b}$, then $\tau$ is the identity map. Otherwise, there exists some $0<j<a$ such that $\varphi^j(\textbf{z})\in\mathcal{Z}_{a,b}$, so $z_{a-j}+z_{a-j+1}+\dots+z_{a-1}\leq jb/a$. Meanwhile, $\textbf{z}\in\mathcal{Z}_{a,b}$ implies $z_{0}+z_{1}+\dots+z_{a-j-1}\leq(a-j-1)b/a$. These two inequalities contradict $\sum_{l=0}^{a-1}z_{l}=b$. Furthermore, $\varphi^j(\textbf{z})=\varphi^k(\textbf{z})$ if and only if $j\equiv k\pmod{a}$ since $a$ and $b$ are coprime.

On the other hand, $$|\mathcal{Z}_{a,b}|=\frac{1}{a+b}\binom{a+b}{a}=\frac{1}{a}\binom{a+b-1}{b}=\frac{1}{a}|\mathcal{Z'}_{a,b}|.$$
Hence each orbit has $a$ elements, among which there is one contained in $\mathcal{Z}_{a,b}$.
\end{proof}

\begin{corollary}
Let $\lambda\in\mathcal{P}_{a,b}$. Fix $\textbf{z}$ and $\textbf{y}$ the associated coordinates of $\lambda$ in $\mathcal{Z}_{a,b}$ and $\mathcal{Y}_{a,b}$, respectively. Then

(1)(Tripathi, \cite{Tripathi2009})
\begin{equation}\label{eq:nopy}
\#(\lambda)=\frac{1}{2}(a-1)(b-1)-\sum_{j=1}^{a-1}y_{j},
\end{equation}
and
\begin{equation}\label{eq:sizey}
\sigma(\lambda)=\frac{1}{24}(a^2-1)(b^2-1)+\frac{1}{2}\Bigg(\sum_{j=1}^{a-1}\Big(ay_{j}^{2}+b(a-1-2j)y_{j}\Big)-(\sum_{j=1}^{a-1}y_j)^2\Bigg).
\end{equation}

(2)\begin{equation}\label{eq:nopz}
\#(\lambda)=\frac{1}{2}(a-1)(b-1)-\sum_{j=0}^{a-2}(a-j-1)z_{j},
\end{equation}
and
\begin{equation}\label{eq:sizez}
\sigma(\lambda)=\frac{1}{24}(a^2-1)(b^2-1)-\frac{1}{2}\Big\{(b-z_0)(a-1)z_0+\sum_{j=1}^{a-2}\Big(\sum_{i=0}^{j-1}(j-i)z_i+\sum_{i=j+1}^{a-1}z_i\Big)(a-j-1)z_{j}\Big\}
\end{equation}

\end{corollary}

\begin{proof}
Recall Theorem \ref{th:bij2}, $\#(\lambda)=\sum_{j=0}^{a-1}n_{j}$ and $\sigma(\lambda)=\sum_{j=0}^{a-1}\Big(jn_{j}+a\binom{n_{j}}{2}\Big)-\binom{\#(\lambda)}{2}$ where $\textbf{n}=\Phi_{1}(\lambda)$.

By the bijection $(5)\to(1)$ in the proof of Lemma \ref{lem:3.2},
\begin{eqnarray*}
  \#(\lambda) &=& \sum_{j=0}^{a-1}n_{j} = \sum_{j=0}^{a-1}n_{jb}= \sum_{j=1}^{a-1}\lfloor jb/a\rfloor-\sum_{j=1}^{a-1}y_{j}\\
  &=& \sum_{j=1}^{a-1}\frac{jb-r_a(jb)}{a}-\sum_{j=1}^{a-1}y_{j}= \frac{b}{a}\sum_{j=1}^{a-1}j-\frac{1}{a}\sum_{j=1}^{a-1}j-\sum_{j=1}^{a-1}y_{j}\\
  &=& \frac{1}{2}(a-1)(b-1)-\sum_{j=1}^{a-1}y_{j},
\end{eqnarray*}
and
\begin{eqnarray*}
  \sigma(\lambda) &=& \sum_{j=0}^{a-1}\Big(jn_{j}+a\binom{n_{j}}{2}\Big)-\binom{\#(\lambda)}{2}\\
  &=& \sum_{j=0}^{a-1}\Big(r_a(jb)n_{jb}+a\binom{n_{jb}}{2}\Big)-\binom{\frac{1}{2}(a-1)(b-1)-\sum_{j=1}^{a-1}y_{j}}{2}\\
  &=& \sum_{j=1}^{a-1}\Big(r_a(jb)(\lfloor jb/a\rfloor-y_{j})+a\binom{\lfloor jb/a\rfloor-y_{j}}{2}\Big)-\binom{\frac{1}{2}(a-1)(b-1)-\sum_{j=1}^{a-1}y_{j}}{2}\\
  &=& \sum_{j=1}^{a-1}\Big(r_a(jb)\lfloor jb/a\rfloor+a\binom{\lfloor jb/a\rfloor}{2}\Big)-\binom{\frac{1}{2}(a-1)(b-1)}{2}\\
  &&+\sum_{j=1}^{a-1}\Big(a\binom{y_{j}+1}{2}-r_a(jb)y_{j}-a\lfloor jb/a\rfloor y_{j}\Big)-\binom{\sum_{j=1}^{a-1}y_{j}+1}{2}\\
  &&+\frac{1}{2}(a-1)(b-1)\sum_{j=1}^{a-1}y_{j}\\
  &=& \sum_{j=1}^{a-1}\Big(\frac{1}{a}r_a(jb)(jb-r_a(jb))+a\binom{jb/a}{2}+a\binom{r_a(jb)/a+1}{2}-r_a(jb)jb/a\Big)\\
  &&-\frac{1}{8}(a-1)(b-1)(ab-a-b-1)+\sum_{j=1}^{a-1}\Big(a\binom{y_{j}+1}{2}-jby_{j}\Big)\\
  &&-\binom{\sum_{j=1}^{a-1}y_{j}+1}{2}+\frac{1}{2}(a-1)(b-1)\sum_{j=1}^{a-1}y_{j}\\
  &=& \frac{b^2-1}{2a}\sum_{j=1}^{a-1}j^2-\frac{b-1}{2}\sum_{j=1}^{a-1}j-\frac{1}{8}(a-1)(b-1)(ab-a-b-1)\\
  &&+\frac{1}{2}\Bigg(\sum_{j=1}^{a-1}\Big(ay_{j}^{2}+b(a-1-2j)y_{j}\Big)-(\sum_{j=1}^{a-1}y_j)^2\Bigg)\\
  &=& \frac{1}{24}(a^2-1)(b^2-1)+\frac{1}{2}\Bigg(\sum_{j=1}^{a-1}\Big(ay_{j}^{2}+b(a-1-2j)y_{j}\Big)-(\sum_{j=1}^{a-1}y_j)^2\Bigg).
\end{eqnarray*}

By the bijection $(4)\to(5)$ in the proof of Lemma \ref{lem:3.2}, we have

\begin{equation*}
  \sum_{j=1}^{a-1}y_{j}=\sum_{j=1}^{a-1}\sum_{l=0}^{j-1}z_{l}=\sum_{j=0}^{a-2}(a-j-1)z_{j},
\end{equation*}

\begin{equation*}
  \sum_{j=1}^{a-1}jy_{j}=\sum_{l=0}^{a-2}\sum_{j=l+1}^{a-1}z_{l}=\sum_{j=0}^{a-2}\frac{1}{2}(a+j)(a-j-1)z_{j},
\end{equation*}
and
\begin{eqnarray*}
  \sum_{j=1}^{a-1}y_{j}^{2}&=&\sum_{l=0}^{a-2}\sum_{j=l+1}^{a-1}z_{l}^2+\sum_{j=2}^{a-1}\sum_{i=0}^{j-2}\sum_{k=i+1}^{j-1}z_{i}z_{k}\\
  &=&\sum_{j=0}^{a-2}(a-j-1)z_{j}^{2}+\sum_{j=1}^{a-2}(\sum_{i=0}^{j-1}z_{i})(a-j-1)z_{j}.
\end{eqnarray*}
Hence,
\begin{equation*}
  \#(\lambda)=\frac{1}{2}(a-1)(b-1)-\sum_{j=0}^{a-2}(a-j-1)z_{j},
\end{equation*}
and
\begin{eqnarray*}
\sigma(\lambda) &=& \frac{1}{24}(a^2-1)(b^2-1)-\frac{1}{2}\Big\{\sum_{j=0}^{a-2}(b-z_{j})(j+1)(a-j-1)z_{j}\\
&&-\sum_{j=1}^{a-2}\Big(\sum_{i=0}^{j-1}(i+1)z_i\Big)(a-j-1)z_{j}\Big\}\\
&=&\frac{1}{24}(a^2-1)(b^2-1)-\frac{1}{2}\Big\{(b-z_0)(a-1)z_0\\
&&+\sum_{j=1}^{a-2}\Big((b-z_j)(j+1)-\sum_{i=0}^{j-1}(i+1)z_i\Big)(a-j-1)z_{j}\Big\}\\
&=&\frac{1}{24}(a^2-1)(b^2-1)-\frac{1}{2}\Big\{(b-z_0)(a-1)z_0\\
&&+\sum_{j=1}^{a-2}\Big(\sum_{i=0}^{j-1}(j-i)z_i+\sum_{i=j+1}^{a-1}z_i\Big)(a-j-1)z_{j}\Big\}.
\end{eqnarray*}
\end{proof}

Tripathi used (\ref{eq:sizey}) to prove that the largest size of partitions that are $a$- and $b$-core partitions is $\frac{1}{24}(a^2-1)(b^2-1)$. By (\ref{eq:nopz}) and (\ref{eq:sizez}), we can also obtain the following corollary.

\begin{corollary}(\cite{Olsson2007, Tripathi2009})
$$max\{\#(\lambda):\lambda\in\mathcal{P}_{a,b}\}=\frac{1}{2}(a-1)(b-1),$$
$$max\{\sigma(\lambda):\lambda\in\mathcal{P}_{a,b}\}=\frac{1}{24}(a^2-1)(b^2-1).$$
\end{corollary}

Baek, Nam and Yu \cite{Baek2019} gave an expression for the cardinality of $\mathcal{P}_{t_1,t_2,\dots,t_r}$ where $\{t_1$, $t_2$, $\cdots$, $t_r\}$ has at least one pair of relatively prime numbers. At a seminar in Penn State University, Nam asked whether it is possible to obtain the cardinality of $\mathcal{P}_{s,s+1,\ldots,s+r}$ by applying this expression. To answer it, we generalize the expression a little bit.

\begin{theorem}\label{th:mul}
Let $c_1$, $c_2$, $\cdots$, $c_n$ be positive integers. For $1\leq i\leq n$, let $d_{i}$ be the positive integer such that $1\leq d_{i}\leq a$ and $bd_{i}+c_{i}\equiv0\pmod{a}$. Then there is a bijection from $\mathcal{P}_{a,b,c_1,\ldots,c_n}$ to $\mathcal{Z}_{a,b}\bigcap\Big\{\bigcap_{i=1}^{r}\{\sum_{l=j}^{j+d_i-1}z_{l}\leq(bd_{i}+c_{i})/a$ for $j\geq0\}\Big\}$.

Furthermore, the cardinality of $\mathcal{P}_{a,b,c_1,\ldots,c_n}$ is
$$\frac{1}{a}\Big|\mathcal{Z'}_{a,b}\bigcap\Big\{\bigcap_{i=1}^{n}\{\sum_{l=j}^{j+d_i-1}z_{l}\leq(bd_{i}+c_{i})/a \text{ for } j\geq0\}\Big\}\Big|.$$
\end{theorem}

The number $\frac{1}{a}$ appears due to Corollary \ref{cor}. We sketch the proof because it is similar to the one in \cite{Baek2019}. One thing has to be mentioned is the case $a|c_{i}$ for some $i$ in which $d_{i}=a$, so $\sum_{l=j}^{j+a-1}z_{l}(=b)\leq(ab+c_i)/a$ holds automatically. On the other hand, an $a$-core is apparently a $c_i$-core. Therefore, it is unnecessary to require $a\nmid c_{i}$.

\begin{remark}
Sometimes, we let $d_{i}=0$ if $a|c_{i}$ since it makes no difference.
\end{remark}

The theorem below follows.

\begin{theorem}\label{th:GH}
Let $C_{r}=\bigcap_{i=1}^{r-1}\{\sum_{l=j}^{j+i-1}z_{l}\leq i+1$ for $j\geq0\}$ for $r\geq2$. Then, we have

(1) there is a bijection between $\mathcal{P}_{s,s+1,\ldots,s+r}$ and $G_{s,r}:=\mathcal{Z}_{s+1,s}\bigcap C_r$;

(2) there is a bijection between $\mathcal{P}_{s,s-1,\ldots,s-r}$ and $H_{s,r}:=\mathcal{Z}_{s-1,s}\bigcap C_r$.

In particular, if $r\geq s$, then $\mathcal{P}_{s,s-1,\ldots,s-r}=\{\emptyset\}$ and $\mathcal{P}_{s,s+1,\ldots,s+r}=\mathcal{P}_{s,s+1,\ldots,2s}.$ Moreover, $|\mathcal{P}_{s,s+1,\ldots,s+r}|=2^{s-1}$ for $r\geq s$.

\end{theorem}

\begin{proof}
(1) Taking $(a,b,c_1,\cdots,c_n)=(s+1,s,s+2,s+3,\cdots,s+r)$ into Theorem \ref{th:mul}, we have $d_{i}=r_{s+1}(i)$ for $1\leq i\leq r-1$. Then there is a bijection from $\mathcal{P}_{s,s+1,\ldots,s+r}$ to $$\mathcal{Z}_{s+1,s}\bigcap\Big\{\bigcap_{i=1}^{r-1}\{\sum_{l=j}^{j+r_{s+1}(i)-1}z_{l}\leq1+i-s\lfloor i/(s+1)\rfloor\text{ for }j\geq0\}\Big\}$$
where
$$\mathcal{Z}_{s+1,s}=\{(z_{j})_{j\geq0}:\text{ }z_{j}\in\mathbb{Z}_{\geq0}\text{, }\sum_{l=j}^{j+s}z_{l}=s\text{ for }j\geq0\text{ and }\sum_{l=0}^{j}z_{l}\leq j\text{ for }0\leq j\leq s-1\}.$$

Notice that $\lfloor i/(s+1)\rfloor<\lfloor (i+s+1)/(s+1)\rfloor$ and $\sum_{l=j}^{j+i-1}z_{l}=s\lfloor i/(s+1)\rfloor+\sum_{l=j}^{j+r_{s+1}(i)-1}z_{l}$. It follows that there is a bijection between $\mathcal{P}_{s,s+1,\ldots,s+r}$ and $G_{s,r}.$

The proof of $(2)$ is similar and $$\mathcal{Z}_{s-1,s}=\{(z_{j})_{j\geq0}:\text{ }z_{j}\in\mathbb{Z}_{\geq0}\text{, }\sum_{l=j}^{j+s-2}z_{l}=s\text{ for }j\geq0\text{ and }\sum_{l=0}^{j}z_{l}\leq j+1\text{ for }0\leq j\leq s-3\}.$$

For any $\textbf{z}\in G_{s,r}(r\geq s)$, $$\sum_{l=0}^{s-1}z_{l}\leq s-1\text{, }\sum_{l=1}^{s}z_{l}=s$$ and $$\sum_{l=j}^{j+i-1}z_{l}=s\lfloor i/(s+1)\rfloor+\sum_{l=j}^{j+r_{s+1}(i)-1}z_{l}\leq s\lfloor i/(s+1)\rfloor+r_{s+1}(i)+1<i+1.$$ Thus, $\mathcal{P}_{s,s+1,\ldots,s+r}=\mathcal{P}_{s,s+1,\ldots,2s}.$ Let $r\to \infty$. Then for any $\lambda\in\mathcal{P}_{s,s+1,\ldots,2s}$, $\beta(\lambda)\subset\{1,2,\cdots,s-1\}$, so $|\mathcal{P}_{s,s+1,\ldots,2s}|=2^{s-1}$.
\end{proof}

In order to count the cardinality of $\mathcal{P}_{s,s+1,\ldots,s+r}$, we still need two important lemmas.

\begin{lemma}\label{lem:3.3}
Let $G_{k}=G_{s,r}\bigcap\{\sum_{l=0}^{k-1}z_{l}=k-1\}\bigcap\{\sum_{l=0}^{j}z_{l}<j$ for $k\leq j\leq s-1\}$ for $1\leq k\leq s$. Then,

(1) $\{G_{k}$, $1\leq k\leq s\}$ is disjoint and $\bigcup_{k=1}^{s}G_{k}=G_{s,r}$;

(2) there is a bijection between $G_{k}$ and $G_{k-1,r}\times H_{s-k+1,r}$.
\end{lemma}

\begin{proof}
(1) By the definition, $\{G_{k}$, $1\leq k\leq s\}$ is disjoint and $\bigcup_{k=1}^{s}G_{k}\subset G_{s,r}$. For any $\textbf{z}\in\mathcal{Z}_{s+1,s}$, $z_0=0$, so $\bigcup_{k=1}^{s}G_{k}\supset G_{s,r}$.

(2) If $r=1$, then $C_{r}=\emptyset$, and
\begin{eqnarray*}
  G_k &=& \{(z_{j})_{j\geq0}:\text{ }z_{j}\in\mathbb{Z}_{\geq0}\text{, }\sum_{l=j}^{j+s}z_{l}=s\text{ for }j\geq0\text{ and }\sum_{l=0}^{j}z_{l}\leq j\text{ for }0\leq j\leq s-1\}\\
  && \bigcap\{\sum_{l=0}^{k-1}z_{l}=k-1\}\bigcap\{\sum_{l=0}^{j}z_{l}<j\text{ for }k\leq j\leq s-1\}\\
  &=& \{(z_{j})_{j\geq0}:\text{ }z_{j}\in\mathbb{Z}_{\geq0}\text{ and }z_{j}=z_{r_{s+1}(j)}\}\\
  &&\bigcap\{\sum_{l=0}^{k-1}z_{l}=k-1\text{, }\sum_{l=0}^{j}z_{l}\leq j\text{ for }0\leq j\leq k-1\}\bigcap\{z_k=0\}\\
  &&\bigcap\{\sum_{l=k+1}^{s}z_{l}=s-k+1\text{, }\sum_{l=k+1}^{k+1+j}z_{l}\leq j+1\text{ for }0\leq j\leq s-k-2\}\\
  &\cong& \mathcal{Z}_{k,k-1}\times\mathcal{Z}_{s-k,s-k+1}\\
  &=& G_{k-1,1}\times H_{s-k+1,1}
\end{eqnarray*}

When $r>1$, let $\textbf{z}\in Z_{s+1,s}$, $\textbf{u}\in Z_{k,k-1}$ and $\textbf{v}\in Z_{s-k,s-k+1}$ such that $z_{k}=0$, $z_{j}=u_{j}$ for $0\leq j\leq k-1$ and $z_{j}=v_{j-k-1}$ for $k+1\leq j\leq s$. It is sufficient to prove that $\textbf{z}\in C_r$ if and only if $\textbf{u}$ and $\textbf{v}$ are both contained in $C_{r}$.

If $\textbf{z}\in C_r$, then for $1\leq i\leq r-1$, sum of any $i$ consecutive coordinates of $\textbf{z}$ is not greater than $i+1$. Let $p$ and $q$ be two nonnegative integers with sum $i$. Then
$$\sum_{l=0}^{p-1}u_l+\sum_{l=0}^{q-1}u_{k-1-l}=\sum_{l=0}^{p-1}z_l+\sum_{l=0}^{q-1}z_{k-1-l}\leq p-1+q+1=i<i+1\text{ for }i\leq k-1,$$
$$\sum_{l=0}^{p-1}v_l+\sum_{l=0}^{q-1}v_{s-k-1-l}=\sum_{l=0}^{p-1}z_{l+k+1}+\sum_{l=0}^{q-1}z_{s-l}\leq p+q+1=i+1\text{ for }i\leq s-k+1.$$
By Theorem \ref{th:GH}, $\mathcal{Z}_{s\pm1,s}\bigcap C_{r}=\mathcal{Z}_{s\pm1,s}\bigcap C_{min\{s,r\}}$. Thus, $\textbf{u}$ and $\textbf{v}$ are both contained in $C_{r}$.

Conversely, if $\textbf{u}$ and $\textbf{v}$ belong to $C_{r}$. Let $p'$ and $q'$ be two nonnegative integers such that $p'+q'+1=i\leq min\{s,r\}$. Then
$$\sum_{l=k-p'}^{k+q'}z_{l}=\sum_{l=1}^{p'}u_{k-l}+\sum_{l=0}^{q'-1}v_{l}\leq p'+1+q'+1=i+1.$$
Thus $\textbf{z}\in C_r$.
\end{proof}

For $\omega\in\mathcal{P}_{t}$, define $\alpha(\omega)=\sum_{j=0}^{t-1}\binom{(\Phi_1(\omega))_{j}}{2}$ and $\gamma(\omega)=\sum_{j=0}^{t-1}j(\Phi_1(\omega))_{j}$.

For $1\leq i\leq s$, let $\textbf{z}\in G_{k}$, $\textbf{u}\in G_{k-1,r}$ and $\textbf{v}\in H_{s-k+1,r}$ such that $z_{k}=0$, $z_{j}=u_{j}$ for $0\leq j\leq k-1$ and $z_{j}=v_{j-k-1}$ for $k+1\leq j\leq s$.

\begin{lemma}\label{lem:3.4}
Denote the corresponding partitions for $\textbf{z}$, $\textbf{u}$ and $\textbf{v}$ by $\lambda$, $\mu$ and $\nu$, respectively. Then

(1)\begin{equation}\label{eq:part}
\#(\lambda)=\#(\mu)+\#(\nu)+s-k.
\end{equation}

(2)\begin{equation}\label{eq:alpha}
\alpha(\lambda)=\alpha(\mu)+\alpha(\nu)+\#(\nu).
\end{equation}

(3) \begin{equation}\label{eq:gamma}
\gamma(\lambda)=(s-k)\#(\lambda)+\gamma(\mu)-\gamma(\nu)+\#(\mu)+\binom{s-k}{2}.
\end{equation}

(4)\begin{align}
\sigma(\lambda)&= (s+1)\alpha(\lambda)+\gamma(\lambda)-(s-k)\#(\lambda)+\sigma(\mu)-k\alpha(\mu)-\gamma(\mu)\nonumber\\
&+\sigma(\nu)-(s-k)\alpha(\nu)-\gamma(\nu)-\#(\mu)\#(\nu)+\binom{s-k+1}{2}.
\label{eq:size}
\end{align}
\end{lemma}

\begin{proof}
(1) By (\ref{eq:nopz}), we have
\begin{eqnarray*}
 \#(\lambda)-\#(\mu)-\#(\nu) &=& \binom{s}{2}-\binom{k-1}{2}-\binom{s-k}{2}-\sum_{j=0}^{s}(s-j)z_{j}\\
 &&+\sum_{j=0}^{k-1}(k-1-j)u_{j}+\sum_{j=0}^{s-k-1}(s-k-1-j)v_{j}\\
 &=& -k^2+k+sk-1-\sum_{j=0}^{k-1}(s-k+1)z_{j}\\
 &=& -k^2+k+sk-1-(s-k+1)(k-1)\\
 &=& s-k.
\end{eqnarray*}

Let $\textbf{n}=\Phi_1(\lambda)$, $\textbf{m}=\Phi_1(\mu)$ and $\textbf{l}=\Phi_1(\nu)$. Then it follows from Lemma \ref{lem:3.2} that

$$n_{s+1-j}=n_{js}=j-1-\sum_{l=0}^{j-1}z_{l}\text{ for }1\leq j\leq s;$$
$$m_{k-j}=m_{j(k-1)}=j-1-\sum_{l=0}^{j-1}u_{l}\text{ for }1\leq j\leq k-1;$$
$$l_{j}=l_{j(s-k+1)}=j-\sum_{l=0}^{j-1}v_{l}\text{ for }1\leq j\leq s-k-1.$$

Moreover

$$n_{s+1-j}=m_{k-j}\text{ for }1\leq j\leq k,\qquad n_{s+1-j}=1+l_{j-k-1}\text{ for }k+1\leq j\leq s.$$

(2) \begin{eqnarray*}
 \alpha(\lambda) &=& \sum_{j=0}^{s}\binom{n_{j}}{2}=\sum_{j=1}^{k}\binom{m_{k-j}}{2}+\sum_{j=k+1}^{s}\binom{1+l_{j-k-1}}{2}\\
 &=& \sum_{j=0}^{k-1}\binom{m_{j}}{2}+\sum_{j=0}^{s-k-1}\binom{l_{j}}{2}+\sum_{j=0}^{s-k-1}l_{j}\\
 &=&\alpha(\mu)+\alpha(\nu)+\#(\nu).
\end{eqnarray*}

(3) \begin{eqnarray*}
 \gamma(\lambda) &=& \sum_{j=0}^{s}jn_{j}=\sum_{j=1}^{s}(s+1-j)n_{s+1-j}\\
 &=& \sum_{j=1}^{k}(s+1-j)m_{k-j}+\sum_{j=k+1}^{s}(s+1-j)(1+l_{j})\\
 &=& \sum_{j=0}^{k-1}(s+1-k+j)m_{j}+\sum_{j=k+1}^{s}(s+1-j)+\sum_{j=0}^{s-k-1}(s-k-j)l_{j}\\
 &=& (s-k+1)\#(\mu)+\gamma(\mu)+(s-k)\#(\nu)-\gamma(\nu)+\binom{s-k+1}{2}\\
 &=& \gamma(\mu)-\gamma(\nu)+\#(\mu)+(s-k)\#(\lambda)+\binom{s-k}{2}.
\end{eqnarray*}

(4) \begin{eqnarray*}
 \sigma(\lambda) &=& (s+1)\alpha(\lambda)+\gamma(\lambda)-\binom{\#(\lambda)}{2}\\
 &=& (s+1)\alpha(\lambda)+\gamma(\lambda)-\binom{\#(\mu)}{2}-\binom{\#(\nu)}{2}-\binom{s-k}{2}\\
 &&-\#(\mu)\#(\nu)-(s-k)(\#(\mu)+\#(\nu))\\
 &=& (s+1)\alpha(\lambda)+\gamma(\lambda)+\sigma(\mu)-k\alpha(\mu)-\gamma(\mu)+\sigma(\nu)-(s-k)\alpha(\nu)-\gamma(\nu)\\
 &&-\#(\mu)\#(\nu)-(s-k)\#(\lambda)+\binom{s-k+1}{2}\\
\end{eqnarray*}

\end{proof}

By Lemma \ref{lem:3.3} and Lemma \ref{lem:3.4}, we obtain the following.

\begin{theorem}\label{th:consecutive}
Let

$ N_{r}(s):=|\mathcal{P}_{s,s+1,\ldots,s+r}|$;

$ M_{r}(s):=max\{\#(\lambda):\lambda\in\mathcal{P}_{s,s+1,\ldots,s+r}\}$;

$ f_{r}(s):=\sum_{\lambda\in\mathcal{P}_{s,s+1,\ldots,s+r}}\#(\lambda)$;

$ g_{r}(s):=\sum_{\lambda\in\mathcal{P}_{s,s+1,\ldots,s+r}}\alpha(\lambda)$;

$ h_{r}(s):=\sum_{\lambda\in\mathcal{P}_{s,s+1,\ldots,s+r}}\gamma(\lambda)$;

$ T_{r}(s):=\sum_{\lambda\in\mathcal{P}_{s,s+1,\ldots,s+r}}\sigma(\lambda)$.

Then\\
(1) (Amdeberhan, \cite{Amdeberhan2015})
$$N_{r}(s)=\sum_{k=1}^{s}N_{r}(s-k)N_{r}(k-r)$$ and $N_{r}(s)=1$ for $s\leq0.$\\
(2) $$M_{r}(s)=M_{r}(s-r)+s-1$$ and $M_{r}(s)=0$ for $s\leq0$.\\
(3) $$f_{r}(s)=\sum_{k=1}^{s-1}\Big(2f_{r}(k)N_{r}(s-k-r)+kN_{r}(s-k-1)N_{r}(k-r+1)\Big)-\sum_{k=1}^{r-1}f_{r}(s+k-r)$$ and $f_{r}(s)=0$ for $s\leq0$.\\
(4) $$g_{r}(s)=\sum_{k=1}^{s-1}\Big(2g_{r}(k)N_{r}(s-k-r)+f_{r}(k-r+1)N_{r}(s-k-1)\Big)-\sum_{k=1}^{r-1}g_{r}(s+k-r)$$ and $g_{r}(s)=0$ for $s\leq0$.\\
(5) \begin{eqnarray*}h_{r}(s)&=& (s-k)f_{r}(s)+\sum_{k=1}^{r-1}h_{r}(s+k-r)+\sum_{k=0}^{s-1}f_{r}(k)N_{r}(s-k-r)\\
 &&+\sum_{k=1}^{s}\binom{s-k}{2}N_{r}(k-1)N_{r}(s-k-r+1)\end{eqnarray*} and $h_{r}(s)=0$ for $s\leq0$.\\
(6) \begin{eqnarray*}T_{r}(s) &=& (s+1)g_{r}(s)+h_{r}(s)-(s-k)f_{r}(s)\\
 &&+\sum_{k=1}^{s}\Bigg(\Big(T_{r}(k-1)-kg_{r}(k-1)-h_{r}(k-1)\Big)N_{r}(s-k+1)\\
 &&+\Big(T_{r}(k-r)-(k-1)g_{r}(k-r)-h_{r}(k-r)\Big)N_{r}(s-k)\\
 &&-f_{r}(k-1)f_{r}(s-k-r+1)+\binom{s-k+1}{2}N_{r}(k-1)N_{r}(s-k-r+1)\Bigg)
 \end{eqnarray*} and $T_{r}(s)=0$ for $s\leq0$.\\
\end{theorem}

Authors of \cite{Yang2015} gave the recurrence relation above for $r=2$. Thus, we have generalized their result.

\begin{proof}[Proof of Theorem \ref{th:consecutive}]
(1) is a straightforward corollary to Lemma \ref{lem:3.3}.

By (\ref{eq:part}), $M_{r}(s)=max_{1\leq k\leq s}\{M_{r}(s-k)+M_{r}(k-r)+k-1\}$ where $M_{r}(s)=0$ for $s\leq0$. We prove (2) by induction on $s$. For $1\leq s\leq r$, recall the $\beta$-set is contained in $\{1,2,\cdots,s-1\}$, so $M_{r}(s)=s-1$.

When $s>r$, assume that $M_{r}(j)=M_{r}(j-r)+j-1$ for any $j<s$. It is obvious that $M_{r}(s)\geq M_{r}(s-r)+s-1$. For $1\leq k\leq s-r$, $M_{r}(s-r)\geq M_{r}(s-r-k)+M_{r}(k-r)+k-1\geq M_{r}(s-k)+M_{r}(k-r)+s-2$. Hence $M_{r}(s-k)+M_{r}(k-r)+k-1<M_{r}(s-r)+s-1.$ Furthermore,
\begin{eqnarray*}
M_{r}(s)&=&max_{s-r+1\leq k\leq s}\{M_{r}(s-k)+M_{r}(k-r)+k-1\}\\
&=&max_{0\leq k\leq r-1}\{M_{r}(k)+M_{r}(s-r-k)+s-k-1\}\\
&=&max_{1\leq k\leq r-1}\{M_{r}(s-r-k)+s-2\}\bigcap\{M_{r}(s-r)+s-1\}
\end{eqnarray*}
Notice that $M_{r}(j)$ is nondecreasing for $j<s$. Therefore, $M_{r}(s)=M_{r}(s-r)+s-1$.

(3) Again, by (\ref{eq:part}),
\begin{eqnarray*}
 f_{r}(s) &=& \sum_{k=1}^{s}\Big(f_{r}(k-1)N_{r}(s-k+1-r)+f_{r}(s-k+1-r)N_{r}(k-1)\\
 &&+(s-k)N_{r}(k-1)N_{r}(s-k+1-r)\Big)\\
 &=& \sum_{k=0}^{s-1}f_{r}(k)N_{r}(s-k-r)+\sum_{k=1}^{s}f_{r}(k-r)N_{r}(s-k)\\
 &&+\sum_{k=1}^{s}(k-1)N_{r}(s-k)N_{r}(k-r)\\
 &=& \Big(\sum_{k=0}^{s-1}+\sum_{k=0}^{s-r}\Big)f_{r}(k)N_{r}(s-k-r)+\sum_{k=1}^{s}(k-1)N_{r}(s-k)N_{r}(k-r)\\
 &=& \sum_{k=1}^{r-1}f_{r}(s+k-r)+2\sum_{k=0}^{s-r}f_{r}(k)N_{r}(s-k-r)+\sum_{k=1}^{s}(k-1)N_{r}(s-k)N_{r}(k-r)\\
 &=& \sum_{k=1}^{s-1}\Big(2f_{r}(k)N_{r}(s-k-r)+kN_{r}(s-k-1)N_{r}(k-r+1)\Big)\\
 &&-\sum_{k=1}^{r-1}f_{r}(s+k-r).
\end{eqnarray*}

(4) By (\ref{eq:alpha}),
\begin{eqnarray*}
 g_{r}(s) &=& \sum_{k=1}^{s}\Big(g_{r}(k-1)N_{r}(s-k-r+1)+g_{r}(s-k-r+1)N_{r}(k-1)\\
 && +f_{r}(s-k-r+1)N_{r}(k-1)\Big)\\
 &=& \sum_{k=1}^{s-1}\Big(2g_{r}(k)N_{r}(s-k-r)+f_{r}(k-r+1)N_{r}(s-k-1)\Big)\\
 &&-\sum_{k=1}^{r-1}g_{r}(s+k-r).
\end{eqnarray*}

(5) By (\ref{eq:gamma}),
\begin{eqnarray*}
 h_{r}(s) &=& (s-k)f_{r}(s)+\sum_{k=1}^{s}\Big((h_{r}(k-1)+f_{r}(k-1))N_{r}(s-k-r+1)\\
 &&-h_{r}(s-k-r+1)N_{r}(k-1)+\binom{s-k}{2}N_{r}(k-1)N_{r}(s-k-r+1)\Big)\\
 &=& (s-k)f_{r}(s)+\sum_{k=1}^{r-1}h_{r}(s+k-r)+\sum_{k=0}^{s-1}f_{r}(k)N_{r}(s-k-r)\\
 &&+\sum_{k=1}^{s}\binom{s-k}{2}N_{r}(k-1)N_{r}(s-k-r+1).
\end{eqnarray*}

(6) By (\ref{eq:size}),
\begin{eqnarray*}
 T_{r}(s) &=& (s+1)g_{r}(s)+h_{r}(s)-(s-k)f_{r}(s)\\
 &&+\sum_{k=1}^{s}\Bigg(\Big(T_{r}(k-1)-kg_{r}(k-1)-h_{r}(k-1)\Big)N_{r}(s-k+1)\\
 &&+\Big(T_{r}(s-k-r+1)-(s-k)g_{r}(s-k-r+1)\\
 &&-h_{r}(s-k-r+1)\Big)N_{r}(k-1)\\
 &&-f_{r}(k-1)f_{r}(s-k-r+1)+\binom{s-k+1}{2}N_{r}(k-1)N_{r}(s-k-r+1)\Bigg)\\
 &=& (s+1)g_{r}(s)+h_{r}(s)-(s-k)f_{r}(s)\\
 &&+\sum_{k=1}^{s}\Bigg(\Big(T_{r}(k-1)-kg_{r}(k-1)-h_{r}(k-1)\Big)N_{r}(s-k+1)\\
 &&+\Big(T_{r}(k-r)-(k-1)g_{r}(k-r)-h_{r}(k-r)\Big)N_{r}(s-k)\\
 &&-f_{r}(k-1)f_{r}(s-k-r+1)+\binom{s-k+1}{2}N_{r}(k-1)N_{r}(s-k-r+1)\Bigg).
\end{eqnarray*}
\end{proof}

\section{core partitions with distinct parts}

Core partitions with distinct parts have attracted many attentions since Amdeberhan \cite{Amdeberhan2015a} found the connection between the number of $(s,s+1)$-core partitions with distinct parts and the $(s+1)$th Fibonacci number $F_{s+1}$. The followings are some known fascinating results.

\begin{theorem}\label{th:known}
Let $\mathcal{D}_{t}$ be the set of core partitions with distinct parts, and $\mathcal{D}_{t_1,t_2,\ldots,t_r}=\bigcap_{i=1}^{r}\mathcal{D}_{t_i}$. Then

(1) $|\mathcal{D}_{s,s+1}|=F_{s+1};$

(2) $\mathcal{D}_{2s+1,2s+3}=4^s.$
\end{theorem}

(1) is conjectured by Amdeberhan \cite{Amdeberhan2015a} and proved in \cite{Straub2016} and \cite{Xiong2015}, while (2) is also conjectured by Amdeberhan \cite{Amdeberhan2015a} and proved in \cite{Yan2016} and \cite{Zaleski2017}.

In this section, we fix $\Phi_1(\lambda)=\textbf{n}$, and our main goal is to count the number of elements of $\mathcal{D}_{t}$ with additional restriction. The lemmas below are essential throughout the final section.

\begin{lemma}\label{lem:4.1}(\cite{Xiong2015})
$\lambda\in\mathcal{D}_{t}$ if and only if $n_{j}n_{j+1}=0$ for $j\geq0$.
\end{lemma}

With the help of the above lemma, we now give a new proof of (1) in Theorem \ref{th:known}.

\begin{proof}[Proof of (1) in Theorem \ref{th:known}]
Let $\lambda\in\mathcal{D}_{s,s+1}$. Taking $a=s+1$ and $b=s$ into Lemma \ref{lem:3.2}, we have $\textbf{n}\in\mathcal{N}_{s+1}$, $n_{j-1}-n_{j}\leq 1$ for $1\leq j\leq s$ and $n_{s}-n_{0}\leq 0$. Moreover, $n_{j}\leq 1$ for $1\leq j\leq s-1$ by Lemma \ref{lem:4.1}. Recall the definition of $P(\lambda)$ in the proof of Theorem \ref{th:mod}. For $0\leq i\leq s-1$, the number of $\lambda$ such that $P(\lambda)=s+1-i$ is $\binom{s-i}{i}$ by the combinatorial fact mentioned in the proof of Theorem \ref{th:numberofparts}. Therefore, $|\mathcal{D}_{s,s+1}|=\sum_{i\geq 0}\binom{s-i}{i}=F_{s+1}$.
\end{proof}

\begin{lemma}\label{lem:4.2}
Let $\lambda$ be the conjugate of a $t$-core partition with distinct parts. Then
\begin{equation}\label{lem:4.2:condition}
1\leq n_1\leq n_{t-1}+1\text{ and }n_{j}\leq max\{n_{j-1},\ n_{j-2}\}\text{ for }2\leq j\leq t-1.
\end{equation}
\end{lemma}

Since $\lambda'$ is in $\mathcal{D}_{t}$, $\#_{j}\geq 1$ for $1\leq j \leq \lambda_{1}$. Hence, $1$ belongs to $\beta(\lambda)$ and $a\in\beta(\lambda)$ only if $a-1$ or $a-2$ is in $\beta(\lambda)$. Thus, we have Lemma \ref{lem:4.2}.

Let $\mathcal{D}_{t;x}$ be the set of $t$-core partitions with distinct parts, in which the largest part is $x$. Let $\mathcal{D'}_{t;l}$ denote the set of $t$-core partitions into exactly $l$ distinct parts. By applying Lemma \ref{lem:4.1} and Lemma \ref{lem:4.2}, we obtain some interesting results.

\begin{theorem}\label{th:both2}

(1) $$|\mathcal{D'}_{t;l}|=\sum_{i\geq 1}\binom{t-i}{i}\binom{l-1}{i-1};$$

(2) $$|\mathcal{D}_{t;x}|=|\mathcal{D}_{t-1;x-1}|+|\mathcal{D}_{t-2;x-1}|.$$
for $2\leq x\leq t$, and
$$|\mathcal{D}_{t;0}|=1\text{ if }t\geq1,\qquad|\mathcal{D}_{t;1}|=1\text{ if }t\geq2$$
$$|\mathcal{D}_{1;x}|=0\text{ if }x\geq1,\qquad|\mathcal{D}_{2;x}|=1\text{ if }x\geq1.$$

In particular, $|\mathcal{D}_{t;x}|=2^{x-1}$ if $1\leq x\leq t/2$.

\end{theorem}

\begin{proof}
(1) Let $\lambda\in\mathcal{D'}_{t;l}.$ Assume $P(\lambda)=t-i$ for some $1\leq i\leq t-1$. Recall $\#(\lambda)=\sum_{j=0}^{t-1}n_j$, so $\lambda$ has $\binom{t-i}{i}\binom{l-1}{i-1}$ different choices. Moreover, $|\mathcal{D'}_{t;l}|=\sum_{i\geq 1}\binom{t-i}{i}\binom{l-1}{i-1}$.

(2) It is easy to check the initial condition. For $t\geq3$, we consider the conjugates of elements in $\mathcal{D}_{t;x}$ and prove that there exists a bijection $\bar{\varphi}$ between $\Phi_1(\mathcal{D}_{t;x})$ and $\Phi_1(\mathcal{D}_{t-1;x-1})\bigcup\Phi_1(\mathcal{D}_{t-2;x-1})$. Let $\lambda'\in\mathcal{D}_{t;x}$. Then $\sum_{j=0}^{t-1}n_{j}=x$ and $\textbf{n}$ satisfies \eqref{lem:4.2:condition}. Let $k=max\{j<t:\ n_{j}>0\}$. Then $2\leq k\leq t-1$ since $x\geq 2$. Now we give the definition of $\bar{\varphi}$.

If $n_{k}=1$ and $n_{k}\leq n_{k-1}$, then we define $\bar{\varphi}(\textbf{n})=\textbf{u}$ such that $$\textbf{u}=(u_{0},u_{1},\cdots,u_{t-2})=(n_{0},n_{1},\cdots,n_{k-1},\cancel{n_{k}},n_{k+1},\cdots,n_{t-1}).$$ Thus, $\sum_{j=0}^{t-2}u_{j}=x-1$ and $\textbf{u}$ satisfies \eqref{lem:4.2:condition}. Hence, the conjugate of $\Phi^{-1}(\textbf{u})$ is in $\mathcal{D}_{t-1;x-1}$.

If $n_{k}=1$ and $n_{k-1}=0$, then we define $\bar{\varphi}(\textbf{n})=\textbf{v}$ such that $$\textbf{v}=(v_{0},v_{1},\cdots,v_{t-3})=(n_{0},n_{1},\cdots,n_{k-2},\cancel{n_{k-1}},\cancel{n_{k}},n_{k+1},\cdots,n_{t-1}).$$ Thus, $\sum_{j=0}^{t-2}v_{j}=x-1$ and $\textbf{v}$ satisfies \eqref{lem:4.2:condition}. Hence, the conjugate of $\Phi^{-1}(\textbf{v})$ is in $\mathcal{D}_{t-2;x-1}$.

If $n_{k}>1$, then $k=t-1$. Otherwise, $n_1=1$ and $n_{j}=0$ or $1$ for any $j$. Since $x\leq t$, we have $x=t$, $t$ is odd and $\textbf{n}=(0,2,0,2,\cdots,0,2)$. Define $\bar{\varphi}(\textbf{n})=\textbf{v}$ such that $\textbf{v}=(v_{0},v_{1},\cdots,v_{t-3})=(0,3,0,2,0,2,\cdots,0,2)$. Then the conjugate of $\Phi^{-1}(\textbf{v})$ is in $\mathcal{D}_{t-2;x-1}$.

Conversely, for any $\textbf{u}$ in $\Phi_1(\mathcal{D}_{t-1;x-1})\bigcup\Phi_1(\mathcal{D}_{t-2;x-1})$, we can easily find a partition $\lambda'$ in $\mathcal{D}_{t;x}$ such that $\bar{\varphi}(\textbf{n})=\textbf{u}$. Let $\mu=\Phi_1^{-1}(\textbf{u})$ and $l=max\{j<t-1:\ u_{j}>0\}$. If $u_1\leq2$ and $\mu\in\mathcal{D}_{t-1;x-1}$, then we have $$\textbf{n}=(u_0,u_1,\cdots,u_{l},1,u_{l+1},\cdots,u_{t-2}).$$ If $u_1\leq2$ and $\mu\in\mathcal{D}_{t-2;x-1}$, then we have $$\textbf{n}=(u_0,u_1,\cdots,u_{l},0,1,u_{l+1},\cdots,u_{t-3}).$$ If $u_1\geq3$, then by the condition \eqref{lem:4.2:condition}, we have $t$ is odd and $\textbf{u}=(0,3,0,2,0,2,\cdots,0,2)$. Thus we have $t$ is odd and $\Phi_1(\lambda)=(0,2,0,2,\cdots,0,2).$

Therefore, the recurrence relation $|\mathcal{D}_{t;x}|=|\mathcal{D}_{t-1;x-1}|+|\mathcal{D}_{t-2;x-1}|$ holds for $2\leq x\leq t$.
We prove $|\mathcal{D}_{t;x}|=2^{x-1}$ if $1\leq x\leq t/2$ by induction on $x$. For $t\geq2$, $|\mathcal{D}_{t;1}|=1$. Assume that claim holds for all $x<y$. $2\leq y\leq t/2$ implies $1\leq y-1\leq (t-2)/2\leq (t-1)/2$. Thus,  $|\mathcal{D}_{t;y}|=|\mathcal{D}_{t-1;y-1}|+|\mathcal{D}_{t-2;y-1}|=2^{y-2}+2^{y-2}=2^{y-1}.$

\end{proof}

\begin{remark}
(1) Our proof above also leads to the result that if $t$ is odd, then $|\mathcal{D}_{t;x}|=|\mathcal{D}_{t-1;x-1}|+|\mathcal{D}_{t-2;x-1}|$ for $2\leq x \leq t+1.$

(2) The following table shows the cardinality of $\mathcal{D}_{t;x}$ for $2\leq t\leq10$ and $0\leq x\leq 10.$
\end{remark}

\begin{center}
\begin{tabular}{|c|c|c|c|c|c|c|c|c|c|c|c|}
  \hline
  $\mathcal{D}_{t;x}$ & x=0 & 1 & 2 & 3 & 4 & 5 & 6 & 7 & 8 & 9 & 10\\\hline
  t=2 & 1 & 1 & 1 & 1 & 1 & 1 & 1 & 1 & 1 & 1 & 1 \\
  3 & 1 & 1 & 1 & 1 & 1 & 1 & 1 & 1 & 1 & 1 & 1 \\
  4 & 1 & 1 & 2 & 2 & 2 & 3 & 3 & 3 & 4 & 4 & 4 \\
  5 & 1 & 1 & 2 & 3 & 3 & 3 & 4 & 5 & 5 & 5 & 6 \\
  6 & 1 & 1 & 2 & 4 & 5 & 5 & 6 & 8 & 10 & 11 & 12 \\
  7 & 1 & 1 & 2 & 4 & 7 & 8 & 8 & 10 & 13 & 17 & 20 \\
  8 & 1 & 1 & 2 & 4 & 8 & 12 & 13 & 14 & 18 & 24 & 31 \\
  9 & 1 & 1 & 2 & 4 & 8 & 15 & 20 & 21 & 24 & 31 & 41 \\
  10 & 1 & 1 & 2 & 4 & 8 & 16 & 27 & 33 & 35 & 42 & 55 \\
  \hline
\end{tabular}
\end{center}

Finally, we end this paper with a fact about $\mathcal{D}_{s,s+1,t}$.

\begin{theorem}\label{th:last}
$\mathcal{D}_{s,s+1,t}=\mathcal{D}_{s,s+1}$ if $t>s+1$.
\end{theorem}

\begin{proof}
It is obvious that $\mathcal{D}_{s,s+1,t}\supset\mathcal{D}_{s,s+1}$. On the other hand, let $\lambda\in\mathcal{D}_{s,s+1}$ with the associated $\textbf{n}$ in $\mathcal{N}_{s+1}$. Then $n_0=n_s=0$, $n_{j}n_{j+1}=0$ and $n_{j}\leq1$ for $j\geq0$, by Lemma \ref{lem:4.1} and Lemma \ref{lem:3.1}. Notice that $\lambda$ is contained in $\mathcal{P}_t$ if $n_{j+t}-n_{j}\leq\lfloor(r_{s+1}(j)+t)/(s+1)\rfloor$. Clearly, it holds since $\lfloor(r_{s+1}(j)+t)/(s+1)\rfloor\geq1\geq n_{j+t}-n_{j}$. Thus $\lambda\in\mathcal{D}_{s,s+1,t}.$
\end{proof}

\section{uncited references}
\cite{Armstrong2014, Aukerman2009, Brylawski1973, Ford2009, Sills2017, Wang2015}

\subsection*{Acknowledgments}
This work was completed when the author visited PSU. He is grateful for their hospitality. This work is supported by the National Natural Science Foundation of China (Grant No. 11571303) and China Scholarship Council. The author also thanks George Andrews and Shane Chern for kind comments.

\bibliographystyle{amsplain}

\end{document}